\documentclass[12]{amsart}
\usepackage{amsmath,amssymb,amsthm,color,enumerate,comment,centernot,enumitem,url,cite}
\usepackage{graphicx,relsize,bm}
\usepackage{mathtools}
\usepackage{array}

\makeatletter
\newcommand{\tpmod}[1]{{\@displayfalse\pmod{#1}}}
\makeatother

\newtheorem{thm}{Theorem}[section]

\theoremstyle{remark}

\theoremstyle{definition}

\newtheorem{rem}[thm]{Remark}

\theoremstyle{THM}

\newcommand{\FF}{{\mathcal F}}

\newcommand{\Z}{{\mathbb Z}}
\newcommand{\Q}{{\mathbb Q}}

\newcommand{\F}{{\mathbb F}}

\makeatletter
\@namedef{subjclassname@2020}{%
  \textup{2020} Mathematics Subject Classification}
\makeatother

\def\red#1 {\textcolor{red}{#1 }}
\def\blue#1 {\textcolor{blue}{#1 }}

\numberwithin{equation}{section}

\begin{document}

\title[Wieferich Primes and Monogenic Trinomials]{Wieferich Primes and Monogenic Trinomials}

%\author{Joshua Harrington}
%\address{Department of Mathematics, Cedar Crest College, Allentown, Pennsylvania, USA}
%\email[Joshua Harrington]{Joshua.Harrington@cedarcrest.edu}

\author{Lenny Jones}
\address{Professor Emeritus, Department of Mathematics, Shippensburg University, Shippensburg, Pennsylvania 17257, USA}
\email[Lenny~Jones]{doctorlennyjones@gmail.com}

%\author{Daniel White}
%\address{Department of Mathematics, Bryn Mawr College, Bryn Mawr, Pennsylvania 19010-2899, USA}
%\email[Daniel~White]{dfwhite@brynmawr.edu}
\date{\today}

\begin{abstract}
A prime $p$ is called a {\em Wieferich prime} if $2^{p-1}\equiv 1 \pmod{p^2}$.   
     A monic polynomial $f(x)\in \Z[x]$ of degree $N\ge 2$ is called {\em monogenic} if $f(x)$ is irreducible over ${\mathbb Q}$ and 
    $\{1,\theta,\theta^2,\ldots,\theta^{N-1}\}$
    is a basis for the ring of integers of ${\mathbb Q}(\theta)$, where $f(\theta)=0$.  In this article, we show that $\FF_p(x):=x^{2p}+2x^{p}+2$ is monogenic if and only if $p$ is not a Wieferich prime. 
\end{abstract}

\subjclass[2020]{Primary  11A41, 11R04; Secondary 11A07, 11R09}
\keywords{Wieferich prime, monogenic, trinomial}

\maketitle
\section{Introduction}\label{Section:Intro}
We say that a prime $p$ is a {\em Wieferich prime} if $2^{p-1}\equiv 1 \pmod{p^2}$. Currently, the only known Wieferich primes are $p=1093$ and $p=3511$. 

We say that a monic polynomial $f(x)\in \Z[x]$ is {\em monogenic} if $f(x)$ is irreducible over $\Q$ and  $\{1,\theta,\theta^2,\ldots,\theta^{\deg(f)-1}\}$
  is a basis for the ring of integers, $\Z_K$, of $K={\mathbb Q}(\theta)$, where $f(\theta)=0$. 
  It is well known \cite{Cohen} that 
  \begin{equation} \label{Eq:Dis-Dis}
\Delta(f)=\left[\Z_K:\Z[\theta]\right]^2\Delta(K),
\end{equation} where $\Delta(f)$ and $\Delta(K)$ denote the discriminants over $\Q$, respectively, of $f(x)$ and the number field $K$. We refer to $\left[\Z_K:\Z[\theta]\right]$ as the {\em index} of $f(x)$, and we denote it index($f$). 
Thus, from \eqref{Eq:Dis-Dis}, 
\begin{equation}\label{Mono}
f(x) \ \mbox{is monogenic if and only if} \ \Delta(f)=\Delta(K),\  \mbox{or equivalently,} \  \Z_K=\Z[\theta].
\end{equation}

In this article, we prove the following theorem: 
\begin{thm}\label{Thm:Main}
Let $p\ge 3$ be a prime, and let   
\[\FF_{p}(x):=x^{2p}+2x^{p}+2.\] Then $\FF_{p}(x)$ is monogenic if and only if $p$ is not a Wieferich prime.
\end{thm}
\noindent 
While analogous theorems exist for $k$-Wall-Sun-Sun primes \cite{Jones k-Wall-Sun-Sun} (aka Fibonacci-Wieferich primes when $k=1$) and generalized Wall-Sun-Sun primes \cite{Jones Gen Wall-Sun-Sun, HJBAMS Gen Wall-Sun-Sun}, 
these previous results do not pertain to the situation in Theorem \ref{Thm:Main}.

\begin{rem}
We point out that Theorem \ref{Thm:Main} is also valid when $p=2$ since 2 is not a Wieferich prime and it is easy to verify that $x^4+2x^2+2$ is monogenic.
\end{rem}

\section{Preliminaries}\label{Section:Prelim}
Throughout this article, we assume that $p$ is an odd prime and that 
\begin{equation}\label{F}
\FF_{p}(x):=x^{2p}+2x^{p}+2.
\end{equation} Note that $\FF_p(x)$ is irreducible over $\Q$ since it is $2$-Eisenstein. % with respect to the prime 2. 
From a theorem due to Swan \cite{Swan}, we have 
\begin{equation}\label{Dis} 
\Delta(\FF_{p})=-2^{3p-1}p^{2p}. 
\end{equation} 
Thus, we see from \eqref{Eq:Dis-Dis} and  \eqref{Mono} that 
\begin{equation}\label{Mono condition}
\FF_p(x) \ \mbox{is monogenic if and only if} \ \gcd(2p,{\rm index}(\FF_p))=1.
\end{equation} One method that can be used, for any monic irreducible polynomial, to determine explicit conditions (more so than \eqref{Mono condition}) on the prime divisors of the polynomial discriminant to guarantee monogenicity is known as Dedekind's Index Criterion \cite{Cohen}. However, another more-algorithmic method exists that is derived from Dedekind's theorem and designed explicitly for trinomials. This streamlined version of Dedekind's theorem is due to Jakhar, Khanduja and Sangwan \cite[Theorem 1.1]{JKS2}, and requires that only one out of a list of five conditions be checked for each prime divisor $p$ of the discriminant of the trinomial to determine whether $p$ divides the index of the trinomial. Moreover, a careful analysis of \cite[Theorem 1.1]{JKS2} reveals that we only need to apply the fourth condition of that theorem in our situation, which yields the following very succinct result.   
\begin{thm}\label{Thm:JKS} 
%Let $\Z_K$ denote the ring of integers of $K=\Q(\theta)$, where $\FF_p(\theta)=0$, so that index($\FF_p$)=\left[\Z_K:\Z[\theta]\right]$. and let.  . 
The trinomial $\FF_p(x)$ is monogenic if and only if the polynomials 
   \begin{equation*}
     H_1(x):=x^{2}+2x+2 \quad \mbox{ and }\quad H_2(x):=\dfrac{2x^{p}+2-\left(2x+2\right)^{p}}{p} 
   \end{equation*}
   are coprime in $\F_p[x]$. 
\end{thm}
\noindent Since 
\begin{align}\label{H2}
\begin{split}
  H_2(x)&=\dfrac{2x^{p}+2-2^p\left(x+1\right)^{p}}{p}\\
  &=\left(\frac{2-2^p}{p}\right)x^p+\frac{2-2^p}{p}-2^p\sum_{j=1}^{p-1}\frac{\binom{p}{j}}{p}x^j, 
  \end{split}
\end{align}
we see indeed that $H_2(x)\in \Z[x]$, by Fermat's Little Theorem and the fact that $\binom{p}{j}\equiv 0 \pmod{p}$ for all $j$ with $1\le j\le p-1$,

\section{The Proof of Theorem \ref{Thm:Main}}
\begin{proof} From Theorem \ref{Thm:JKS}, we must determine exactly when $H_1(x)$ and $H_2(x)$  
   are coprime in $\F_p[x]$. Observe from \eqref{H2} that  
    \[\deg(H_2)\ge p-1\ge 2=\deg(H_1).\] 
     in $\F_p[x]$. Thus, our general strategy is to determine  the conditions under which $pH_2(\theta)\equiv 0 \pmod{p^2}$ for the zeros $\theta$ of $H_1(x)$, which will tell us precisely when $H_1(x)$ and $H_2(x)$ are not coprime in $\F_p[x]$, and consequently, when $\FF_p(x)$ is not monogenic, by Theorem \ref{Thm:JKS}. We split the proof into the two cases $p\equiv \pm 1\pmod{4}$.
   
 Suppose first that $p\equiv 3 \pmod{4}$. 
   Then $H_1(x)$ is irreducible in $\F_p[x]$ since $-1$ is not a square in $\F_p$. Hence, if $H_1(x)$ and $H_2(x)$ are not coprime in $\F_p[x]$, it must be that $H_1(x)$ is an irreducible factor of $H_2(x)$, so that every zero of $H_1(x)$ is a zero of $H_2(x)$. Therefore, it is enough to examine $pH_2(\theta)$ modulo $p^2$ for any zero $\theta$ of $H_1(x)$.  Observe that $\alpha:=-1+i$ is a zero of $H_1(x)$, where $i\not\in \F_p$ with $i^2=-1$. Thus, to establish the theorem in this case, we show that 
      \begin{equation}\label{Main Thm for p=3 mod 4}
     pH_2(\alpha)\equiv 0 \pmod{p^2} \quad \mbox{if and only if} \quad 2^{p-1}\equiv 1 \pmod{p^2}.
   \end{equation} 
     Since $(\alpha+1)^p=i^p=-i$, we have from the definition of $H_2(x)$, that 
   \begin{equation}\label{pH2(alpha)1 p mod 4=3}
   pH_2(\alpha)=2\alpha^p+2-2^p(\alpha+1)^p=2\alpha^p+2+2^pi.
   \end{equation} 
    A straightforward induction argument shows that 
   \[\alpha^j=(-1)^{\frac{j-3}{4}}2^{\frac{j-1}{2}}(1+i) \quad \mbox{if $j\equiv 3 \pmod{4}$}.\] Thus, 
   \[\alpha^p=(-1)^{\frac{p-3}{4}}2^{\frac{p-1}{2}}(1+i),\] which yields the following rearrangement 
   of $pH_2(\alpha)$ from \eqref{pH2(alpha)1 p mod 4=3}: 
   \begin{align}\label{pH2(alpha)2 p mod 4=3}
   \begin{split}
   pH_2(\alpha)&=2\left((-1)^{\frac{p-3}{4}}2^{\frac{p-1}{2}}(1+i)\right)+2+2^pi\\
   &=2+(-1)^{\frac{p-3}{4}}2^{\frac{p+1}{2}}+\left(2^p+(-1)^{\frac{p-3}{4}}2^{\frac{p+1}{2}}\right)i.
   \end{split}
   \end{align}
   Consequently, from \eqref{pH2(alpha)2 p mod 4=3}, it follows that 
   $pH_2(\alpha)\equiv 0 \pmod{p^2}$ if and only if 
    \begin{equation}\label{H2(alpha)}
    2+(-1)^{\frac{p-3}{4}}2^{\frac{p+1}{2}} \equiv 0 \pmod{p^2} \quad \mbox{and} \quad 2^p+(-1)^{\frac{p-3}{4}}2^{\frac{p+1}{2}} \equiv 0 \pmod{p^2}.
    \end{equation}
    
    To prove \eqref{Main Thm for p=3 mod 4}, suppose first that $pH_2(\alpha)\equiv 0 \pmod{p^2}$. From the first congruence in \eqref{H2(alpha)} we see that
    \begin{equation}\label{Con1}
    (-1)^{\frac{p-3}{4}}2^{\frac{p-1}{2}}\equiv -1 \pmod{p^2}.
    \end{equation} Squaring both sides of \eqref{Con1}  yields $2^{p-1}\equiv 1 \pmod{p^2}$. 
    
    Conversely, suppose that $2^{p-1}\equiv 1 \pmod{p^2}$. Then
    \begin{equation}\label{factorization}
    2^{p-1}-1\equiv (2^{\frac{p-1}{2}}-1)(2^{\frac{p-1}{2}}+1)\equiv 0 \pmod{p^2}.
    \end{equation} From \eqref{factorization}  we conclude that  
    \[2^{\frac{p-1}{2}}\equiv \pm 1 \pmod{p^2},\] since $p$ divides one and only one of the factors in the factorization \eqref{factorization}. 
        We claim that 
    \begin{equation}\label{claim}
      2^{\frac{p-1}{2}}\equiv (-1)^{\frac{p+1}{4}} \pmod{p^2}.   
    \end{equation}  
   Suppose first that $2^{\frac{p-1}{2}}\equiv 1 \pmod{p^2}$, so that $2^{\frac{p-1}{2}}\equiv 1 \pmod{p}$. By Euler's criterion and the quadratic character of 2, we have that 
      \[2^{\frac{p-1}{2}}\equiv \left(\frac{2}{p}\right)\equiv (-1)^{\frac{p^2-1}{8}} \pmod{p}.\] Since $p\equiv 3\pmod{4}$, it follows that $p\equiv 7 \pmod{8}$ and thus, 
      \[(-1)^{\frac{p^2-1}{8}}=(-1)^{\frac{p+1}{4}}=1.\] Similarly, if $2^{\frac{p-1}{2}}\equiv -1 \pmod{p^2}$, then $p\equiv 3 \pmod{8}$, so that  
      \[(-1)^{\frac{p^2-1}{8}}=(-1)^{\frac{p+1}{4}}=-1.\] Therefore, the claim \eqref{claim} is estabished. 
      Therefore, since $2^p\equiv 2\pmod{p^2}$ and $p\equiv 3 \pmod{4}$, we have from \eqref{pH2(alpha)2 p mod 4=3} and  \eqref{claim} that   
    \begin{align*}\label{pH2(alpha)2}
    pH_2(\alpha)&\equiv 2\left(1+(-1)^{\frac{p-3}{4}}2^{\frac{p-1}{2}}\right)(1+i)\\
    &\equiv 2\left(1+(-1)^{\frac{p-3}{4}}(-1)^{\frac{p+1}{4}}\right)(1+i)\\
    &\equiv  2\left(1+(-1)^{\frac{p-1}{2}}\right)(1+i)\\
    &\equiv 0 \pmod{p^2},
    \end{align*} 
    which finishes the proof of \eqref{Main Thm for p=3 mod 4}, and completes the proof of the theorem when $p\equiv 3 \pmod{4}$.
    
    Suppose next that $p\equiv 1 \pmod{4}$. Then $H_1(x)$ is reducible in $\F_p[x]$ since $-1$ is a square in $\F_p$. Let $\alpha=-1+i\in \F_p$ be a zero of $H_2(x)$, where $i$ and $\alpha$ are such that $1\le \alpha\le p-3$ and $2\le i\le p-2$ with $i^2\equiv -1 \pmod{p}$. Then, $\beta=p-1-i\in \F_p[x]$ is the other zero of $H_1(x)$, where $1\le \beta\le p-3$. The strategy here, as in the previous case, is to determine conditions under which $pH_2(\theta)\equiv 0 \pmod{p^2}$, where $\theta\in \{\alpha,\beta\}$. It is conceivable that exactly one of the two zeros $\theta$, and not the other, could satisfy the congruence $pH_2(\theta)\equiv 0 \pmod{p^2}$. However, it turns out that 
    \begin{equation}\label{Equivalent}
    pH_2(\alpha)\equiv 0 \pmod{p^2} \quad \mbox{if and only if} \quad pH_2(\beta)\equiv 0 \pmod{p^2},
    \end{equation} which shows that considering the single zero $\alpha$ will suffice. We postpone the proof of \eqref{Equivalent} to focus first on $\alpha$, and show that 
    \begin{equation}\label{Main Thm p mod 4=1}
    pH_2(\alpha)\equiv 0 \pmod{p^2} \quad \mbox{if and only if} \quad 2^{p-1}\equiv 1 \pmod{p^2}.
   \end{equation} Since $H_1(\alpha)\equiv 0 \pmod{p}$, we see that $\alpha^2\equiv -2(\alpha+1) \pmod{p}$, which implies % that
    \begin{equation}\label{alpha^{2p}}
    \alpha^{2p}\equiv -2^p(\alpha+1)^p\equiv -2^pi^p \pmod{p^2}.
     \end{equation}  Therefore,  from the definition of $H_2(x)$, \eqref{F} and \eqref{alpha^{2p}}, it follows that
        \begin{align}\label{pH2=F}
        \begin{split}
pH_2(\alpha)&=2\alpha^p+2-2^p(\alpha+1)^p\\
          &=2\alpha^p+2-2^pi^p\\
          &\equiv 2\alpha^p+2+\alpha^{2p}\pmod{p^2}\\
          &\equiv \FF_p(\alpha) \pmod{p^2}. 
       \end{split}    
        \end{align}

        To prove \eqref{Main Thm p mod 4=1}, suppose that $pH_2(\alpha)\equiv 0 \pmod{p^2}$. Then $\FF_p(\alpha)\equiv 0\pmod{p^2}$ from \eqref{pH2=F}, and hence 
        \begin{equation}\label{alpha^p}
        \alpha^{p}\equiv -\frac{\alpha^{2p}+2}{2} \pmod{p^2}.
        \end{equation} Squaring both sides of \eqref{alpha^p} yields 
        \begin{equation}\label{Squared}
        \alpha^{2p}\equiv \frac{(\alpha^{2p}+2)^2}{4} \pmod{p^2}.
        \end{equation} Thus, from \eqref{alpha^{2p}} and \eqref{Squared}, we have 
        \[-2^{p+2}i^p\equiv 2^{2p}i^{2p}-2^{p+2}i^p+4\pmod{p^2},\] or equivalently,
        \begin{equation}\label{Eq}
        2^{2p}i^{2p}+4\equiv 0 \pmod{p^2}.
        \end{equation}    
    Since $i^2\equiv -1 \pmod{p}$, we have    
    \begin{equation}\label{i^{2p}}
     i^{2p}\equiv (-1)^p\equiv -1 \pmod{p^2}.
    \end{equation} Consequently, from \eqref{Eq} and \eqref{i^{2p}}, we deduce that 
    \begin{equation}\label{Wie}
    (2^p-2)(2^p+2)\equiv 0 \pmod{p^2}.
     \end{equation}   Since $2^p\equiv 2 \pmod{p}$, we have that $2^p+2\not \equiv 0 \pmod{p}$ since $p\ge 3$. Thus, it follows from \eqref{Wie} that $2^p-2\equiv 0 \pmod{p^2}$ and, therefore, $2^{p-1}\equiv 1 \pmod{p^2}$. % $p$ is a Weiferich prime.  
     
     Conversely, if $2^{p-1}\equiv 1 \pmod{p^2}$, then $2^p-2\equiv 0\pmod{p^2}$, and the steps of the proof in the other direction can be reversed from \eqref{Wie} back to \eqref{Squared} to give us
     \begin{equation}\label{Reverse}
     \left(\alpha^{2p}+2-2\alpha^p\right)\left(\alpha^{2p}+2+2\alpha^p\right)\equiv 0\pmod{p^2}.
     \end{equation} Suppose that 
     \[\alpha^{2p}+2-2\alpha^p\equiv \left(\alpha^{2}+2-2\alpha\right)^p\equiv 0 \pmod{p}.\] Then, since $H_1(\alpha)=\alpha^2+2\alpha+2\equiv 0 \pmod{p}$, we arrive at the contradiction  
     \[H_1(\alpha)-\left(\alpha^{2}+2-2\alpha\right)=4\alpha\equiv 0\pmod{p}.\] Therefore, we conclude from \eqref{Reverse} and \eqref{pH2=F} that
     \[pH_2(\alpha)\equiv \FF_p(\alpha) \equiv \alpha^{2p}+2+2\alpha^p \equiv 0 \pmod{p^2},\]
     which completes the proof in this direction for the zero $\alpha$.
     
     To show that we do not need to consider the zero $\beta$, we now provide a proof of \eqref{Equivalent}. Since $\alpha\beta\equiv 2\pmod{p}$, it follows that 
     \begin{equation}\label{alpha beta}
     \alpha^p\beta^p\equiv 2^p\pmod{p^2}.
     \end{equation} Since $H_1(\beta)\equiv 0 \pmod{p}$, we see that $\beta^2\equiv -2(\beta+1) \pmod{p}$, which implies % that
    \begin{equation}\label{beta^{2p}}
    \beta^{2p}\equiv -2^p(\beta+1)^p\equiv -2^p(p-i)^p\equiv 2^pi^p \pmod{p^2}.
     \end{equation}  Therefore,  from the definition of $H_2(x)$, \eqref{F} and \eqref{beta^{2p}}, it follows that
        \begin{align}\label{pH2=F beta}
        \begin{split}
          pH_2(\beta)&=2\beta^p+2-2^p(\beta+1)^p\\
          &=2\beta^p+2-2^p(p-i)^p\\
          &=2\beta^p+2+2^pi^p \pmod{p^2}\\
          &\equiv 2\beta^p+2+\beta^{2p}\pmod{p^2}\\
          &\equiv \FF_p(\beta) \pmod{p^2}. 
       \end{split}    
        \end{align} To establish \eqref{Equivalent}, suppose that $pH_2(\alpha)\equiv 0 \pmod{p^2}$. Then $\FF_p(\alpha)\equiv 0 \pmod{p^2}$ from \eqref{pH2=F}, and therefore, $2^{p}\equiv 2\pmod{p^2}$ from the argument following \eqref{Wie}. Hence, from \eqref{alpha beta}, it follows that 
        \begin{align*}
          \FF_p(\beta)&=\beta^{2p}+2\beta^p+2\\ 
          &\equiv \left(\frac{2^p}{\alpha^p}\right)^2+2\left(\frac{2^p}{\alpha^p}\right)+2\pmod{p^2}\\
          &\equiv \frac{2\left(\alpha^{2p}+2\alpha^p+2\right)}{\alpha^{2p}}\pmod{p^2}\\
          &\equiv \frac{2\FF_p(\alpha)}{\alpha^{2p}}\pmod{p^2}\\
          &\equiv 0 \pmod{p^2}.
        \end{align*}  Thus, $pH_2(\beta)\equiv 0 \pmod{p^2}$ from \eqref{pH2=F beta}. The converse is similar and we omit the details. 
  \end{proof}

%\section*{Acknowledgments} %The author thanks the referee for the valuable suggestions.
%\section{Final Remarks}

% For alignments use AmS-LaTeX constructions not \eqnarray.

%% - theorems and proofs
%\begin{thm}[optional text]
% The optional material will be typeset as part of the theorem heading
%\end{thm}

%\begin{proof}[Optional proof heading]
% the proof
%\end{proof}
% An end-of-proof symbol (open box) will be typeset at the
% end of the proof.

%\ack % or \acks
% Put acknowledgements here
%The authors thank the anonymous referee for the suggestions that helped to improve the paper. %the helpful comments.

%\nocite{*}

% alteratively, bibliographies prepared with BibTeX can be included by
% means of the following commands
%\bibliographystyle{srtnumbered}
%\bibliography{mybib}

\end{document}